\documentclass[11pt,a4paper]{amsart}

\usepackage[utf8]{inputenc}
\usepackage[T1]{fontenc}
\usepackage{amsmath}
\usepackage{amsfonts}
\usepackage{amssymb}
\usepackage{ae}
\usepackage{units}
\usepackage{icomma}
\usepackage{color}
\usepackage{graphicx}
\usepackage{bbm}
\usepackage{amsthm}
\usepackage{amssymb}
\usepackage{cancel}
\usepackage{upgreek}
\usepackage{type1cm}
\usepackage{eso-pic}
\usepackage[breaklinks,pdfpagelabels=false]{hyperref}
\usepackage{xspace}
\usepackage{verbatim}
\usepackage{enumitem}

\newcommand{\Z}{\ensuremath{\mathbbm{Z}}}

\newcommand{\R}{\ensuremath{\mathbbm{R}}}

\newcommand{\abs}[1]{\ensuremath{\left| #1 \right|}}

\renewcommand{\epsilon}{\varepsilon}

\def\qed{\hspace{\stretch1}\ensuremath\square}

\newtheorem{theorem}{Theorem}[section]
\newtheorem{lemma}[theorem]{Lemma}
\newtheorem{proposition}[theorem]{Proposition}

\theoremstyle{definition}
\newtheorem{remark}[theorem]{Remark}

\begin{document}

\title[uniqueness of solutions to random jigsaw puzzles]{A linear threshold for uniqueness of solutions to random jigsaw puzzles}
\author{Anders Martinsson}
\address{Institute of Theoretical Computer Science, ETH Z\"urich, 8092 Z\"urich, Switzerland}
\email{maanders@inf.ethz.ch}
\subjclass[2010]{60C05,05C10}

\begin{abstract}
We consider a problem introduced by Mossel and Ross [\textit{Shotgun assembly of labeled graphs}, \href{http://arxiv.org/abs/1504.07682}{arXiv:1504.07682}]. Suppose a random $n\times n$ jigsaw puzzle is constructed by independently and uniformly choosing the shape of each ``jig'' from $q$ possibilities. We are given the shuffled pieces. Then, depending on $q$, what is the probability that we can reassemble the puzzle uniquely? We say that two solutions of a puzzle are similar if they only differ by a global rotation of the puzzle, permutation of duplicate pieces, and rotation of rotationally symmetric pieces. In this paper, we show that, with high probability, such a puzzle has at least two non-similar solutions when $2\leq q \leq \frac{2}{\sqrt{e}}n$, all solutions are similar when $q\geq (2+\varepsilon)n$, and the solution is unique when $q=\omega(n)$.
\end{abstract}
\maketitle

\section{Introduction}

A \emph{jigsaw puzzle} is a collection of square pieces where each of the four edges of a piece has a shape, referred to as a \emph{jig}, so that it fits together with a subset of the edges of the other pieces. An \emph{edge-matching puzzle} is a collection of square pieces where each side of every piece is given a color. The goal of the respective puzzles is to assemble the pieces into a certain form, in this paper this will always be an $n\times n$ square, such that all pairs of adjacent pieces fit together. In the case of a jigsaw puzzle, this means that the jigs of edges that are aligned next to each other should have complementary shapes, and for an edge-matching puzzle, such pairs of edges should have the same color. Here we assume that the pieces are allowed to be rotated, but not flipped upside down.

In order to make this a bit more formal, we assume that there are $q$ possible types of jigs, enumerated from $1$ to $q$. A jig type can either be symmetric, so that it fits together with itself, or be part of a pair of complementary types that fit together with each other. We can describe this relation by fixing a map $\iota$ from $\{1, 2, \dots, q\}$ to itself such that $\iota \circ \iota = id$ and a jig of type $j$ fits with jigs of type $\iota(j)$. Note that we can consider an edge-matching puzzle as a special case of a jigsaw puzzle by taking $\iota$ equal to the identity map.

In a recent paper by Mossel and Ross \cite{MR15}, a simple model for random edge-matching puzzles, later generalized to random jigsaw puzzles in \cite{BFM16}, was proposed. We imagine that we start with an $n\times n$ grid of identical unit squares. For each of the four sides of each piece, we choose its color, or type of jig respectively, out of $q$ possibilities, under the restriction that pairs of connected sides must get the same color/complementary jig types. Note that this means that, unlike most real jigsaw puzzles, also edges along the boundary are assigned colors/jigs. We will refer to such an assignment of colors to a puzzle as a \emph{coloring}, and an assignment of jig types as a \emph{carving}. The underlying probability space is the set $\Omega$ of the $q^{2n(n+1)}$ possible colorings/carvings of the initial configuration of pieces. Mossel and Ross asked, suppose we choose $\omega\in\Omega$ uniformly at random, then what is the probability that the puzzle can be uniquely recovered from the collection of shuffled pieces? Further, how can this recovery be done efficiently? They called this problem \emph{shotgun assembly of a random jigsaw puzzle} due to its similarity to genetic shotgun sequencing, which is a technique for sequencing a long DNA strand by sampling short random subsequences.

The notion of ``unique recovery'' needs some elaboration. We consider a solution of the puzzle to be a positioning and orientation of the pieces into a fixed $n\times n$ grid such that all adjacent pieces fit together. The solution consisting of all original positions and orientations will be referred to as the \emph{planted assembly}. We here assume that, besides the choice of jigs/colors, all pieces are identical and rotationally symmetric. Hence, given the collection of shuffled jigsaw pieces, any solution to the puzzle is equally likely to be the planted assembly. As the pieces can be rotated, the best we could ever hope for is to be able to recover the planted assembly up to a global rotation of the puzzle. Besides that, the puzzle may contain duplicate pieces, that is, two pieces with identical colors/jigs, or rotationally symmetric pieces, that is, opposite sides have the same colors/jig types. Note that such pieces automatically give rise to additional, albeit not very different, solutions to the puzzle.

We say that two solutions of a puzzle are \emph{similar} if they only differ by a global rotation, permutation of duplicate pieces, and rotation of rotationally symmetric pieces, or, equivalently, if the solutions have the same coloring/carving up to global rotation. In the terminology of \cite{BFM16}, a puzzle has \emph{unique vertex assembly} (UVA) if the only solutions are the four global rotations of the planted assembly, and a puzzle has \emph{unique edge assembly} (UEA) if all solutions are similar.

It was shown by Mossel and Ross \cite{MR15} that, with high probability, a random edge-matching puzzle as above has at least two non-similar solutions when $2\leq q=o(n^{2/3})$, and has a unique solution up to global rotation when $q=\omega(n^2)$. Recently, two papers, Bordenave, Feige and Mossel \cite{BFM16}, and Nenadov, Pfister and Steger \cite{NPS16}, considering this problem were published on arxiv.org, both on May 11:th 2016, and both proving essentially the same result: for $q \geq n^{1+\varepsilon}$, a random edge-matching puzzle has a unique solution up to global rotation with high probability for any fixed $\varepsilon>0$, and for $q=o(n)$ there are duplicate pieces with high probability and hence multiple, but possibly all similar, solutions. It should be noted that the paper by Bordenave et al. assumes that the pieces are not allowed to be rotated, but remarks in the last section of the paper how their argument can be modified slightly, both to allow rotations and to generalize to the random jigsaw puzzle model above.

Concerning the problem of how to recover the planted assembly efficiently, Bordenave et al. describe an algorithm that recovers it with high probability when $q\geq n^{1+\varepsilon}$ with time complexity $n^{O(1/\varepsilon)}$. As a comparison, the general problems of finding one solution to a given $n\times n$ jigsaw puzzle or edge-matching puzzle are  known to be NP-complete \cite{DD07}, see also \cite{BDDHMY16+}. The problem also seems to be hard in practice. In the summer of 2007, a famous edge-matching puzzle, Eternity II, was released, with a \$$2$ million prize for the first complete solution \cite{M07}. This puzzle consists of $256$ square pieces that should be assembled into a $16\times 16$ square. There are in total $22$ edge colors, not including the boundary, which is marked in gray. The competition ended on 31 December 2010, with no solution being found, and at the time of writing, the puzzle is claimed to remain unsolved.

The aim of this paper is to prove the following result regarding uniqueness of the solution of a random jigsaw or edge-matching puzzle.
\begin{theorem} \label{thm:main} As $n\rightarrow \infty$ the following holds with high probability for a random jigsaw puzzle with $q$ types of jigs or random edge-matching puzzle with $q$ colors.
\begin{enumerate}[label=(\roman*)]
\item For $2\leq q \leq \frac{2}{\sqrt{e}}n$, there are at least two non-similar solutions.
\item For $q \geq (2+\varepsilon)n$, for any fixed $\varepsilon>0$, all solutions are similar.
\item For $q=\omega(n)$, the solution is unique up to global rotation.
\end{enumerate}
\end{theorem}
We remark that a weaker form of $(i)$ was first proved in a earlier version of this paper \cite{M16}. This will be shown again in this paper, but using a significantly simpler argument.

The question remains what happens in the interval $\frac{2}{\sqrt{e}} n\leq q \leq 2n$. To try to get a qualitative understanding for this range, we can compare our random model to a collection of $n^2$ independently colored/carved pieces where each color/jig type is chosen uniformly at random, that is, without a planted solution. There are $4^{n^2} (n^2)!$ ways to place and orient the pieces into an $n\times n$ grid, and the probability that the pieces fit together in a given configuration is $q^{-2n(n-1)}$. Hence the expected number of solutions of such a puzzle is
$ \frac{4^{n^2} (n^2)!}{q^{2n(n-1)}} \approx \left(\frac{2n}{\sqrt{e} q} \right)^{2n^2}.$ We see that there is a transition at $q=\frac{2}{\sqrt{e}}n$ where the expected number of solutions goes from being very large to very small.

Connecting this back to our model, the fact that we force the puzzle to have at least one solution may increase the probability that other ways to assemble the pieces are also solutions. On the other hand, based on the proofs in this article, it is my suspicion that typical solutions are either similar to the planted assembly, or have very little similarity to it, and hence this effect should be small. Because of this, I conjecture that the event that all solutions to a random jigsaw or edge-matching puzzle are similar has a sharp threshold at $\frac{2}{\sqrt{e}}n+o(n)$. Moreover, considering how strongly the estimates for $\mathbb{P}(UEA)$ in Section \ref{sec:twooversqrte} depend on $q$, I believe that this threshold is very sharp. It might even jump directly from $o(1)$ to $1-o(1)$ when increasing the number of colors/jig types by $1$.

Using a similar heuristic, we can explain the discrepancy between the bounds in parts $(i)$ and $(ii)$ of Theorem \ref{thm:main} (the factor $\sqrt{e}$). In the proof of $(i)$ we consider the entire assembled puzzle, so, heuristically, random solutions should stop dominating the analysis at $q=\frac{2}{\sqrt{e}}n$. On the other hand, the proof of $(ii)$ is based on considering local solutions to the puzzle. The expected number of ways to build, say, a $k\times k$ square of matching pieces out of $n^2$ independently chosen pieces for $k=o(n)$ is roughly $\frac{ 4^{k^2} (n^2)^{k^2} }{q^{2k(k-1)}}\approx \left(\frac{2n}{q}\right)^{2k^2}$. Hence, in this case, solutions unrelated to the planted assembly should stop dominating the analysis only at $q=2n$. It appears that new ideas are needed to close this gap in the main result.

The remainder of the paper will be structured as follows. In Section \ref{sec:twooversqrte} we give a proof of part $(i)$ of Theorem \ref{thm:main}. Section \ref{sec:dupes} briefly investigates the probability of duplicates and rotationally symmetric pieces, which shows that $(ii)\implies (iii)$. Finally, the proof of part $(ii)$ is split into Sections \ref{sec:allsimilar1} and \ref{sec:allsimilar2}. All aforementioned sections will be formulated in terms of the random jigsaw puzzle model, but as we already noted, the random edge-matching puzzle can be considered as a special case of this.

\section{Proof of Theorem 1.1, part $(i)$}\label{sec:twooversqrte}

The statement follows from a counting argument. Let us refer to the unordered collection of jigsaw pieces of a puzzle as the \emph{box} of the puzzle. That is, the box contains the information of how many pieces of each combination of jig types there are in the puzzle, but no information beyond that about their locations or orientations in the planted assembly. The central observation is that, in the range $2\leq q \leq \frac{2}{\sqrt{e}}$, there are many fewer possible boxes than possible carvings. Thus, a typical box has many solutions.

In order to make this more precise, let $\Omega_{UEA}$ be the set of carvings for which the puzzle has a unique edge assembly. Then for any $\omega\in\Omega_{UEA}$, there are at most $3$ more carvings that yield the same box, namely the $90^\circ, 180^\circ$ and $270^\circ$ rotations of $\omega$. Thus $\abs{\Omega_{UEA}}$ is at most $4$ times the number of possible boxes.

There are $q^4$ ways to choose the jigs of a jigsaw piece: $q$ of which being invariant under a $90^\circ$ rotation, $q^2-q$ being invariant under a $180^\circ$ but not a $90^\circ$ rotation, and the remaining $q^4-q^2$ having no rotational symmetry. Hence there are $\frac{q^4-q^2}{4}+\frac{q^2-q}{2}+q = \frac{q^4+q^2+2q}{4}$ possible types of jigsaw pieces. From this it follows that the total number of boxes (including those without solutions) is ${ \frac{1}{4}(q^4+q^2+2q) + n^2-1 \choose n^2}$.

As the probability of each carving is $q^{-2n(n+1)}$, it follows that
\begin{align*}
\mathbb{P}(UEA) &=  q^{-2n(n+1)}\cdot \abs{\Omega_{UEA}} \leq q^{-2n(n+1)} \cdot 4 { \frac{1}{4}(q^4+q^2+2q) + n^2-1 \choose n^2}\\
&\leq 4 q^{-2n(n+1)} \frac{ \left( \frac{1}{4}(q^4+q^2+2q) + n^2-1 \right)^{n^2}}{ n^2!}\\
&= \Theta\left(\frac{1}{n}\right) q^{-2n} \left( \frac{e (q^4+q^2+2q+4n^2-4)}{4 q^2 n^2}\right)^{n^2},
\end{align*}
where we used Stirling's formula in the last step. One can observe that $$\frac{e (q^4+q^2+2q+4n^2-4)}{4 q^2 n^2}$$ is convex in $q$, equals $\frac{e}{4}+O(\frac{1}{n^2})$ for $q=2$ and $1+O(\frac{1}{n^2})$ for $q=\frac{2}{\sqrt{e}}n$. Hence, for any $2\leq q \leq \frac{2}{\sqrt{e}}n$,
\begin{align*}
\mathbb{P}(UEA) \leq \Theta\left(\frac{1}{n}\right) q^{-2n} \left(1+O\left(\frac{1}{n^2}\right)\right)^{n^2} = \Theta\left(\frac{1}{n}\right) q^{-2n},
\end{align*}
which tends to $0$ as $n\rightarrow\infty$. \qed

\section{Proof of Theorem 1.1, $(ii)$ implies $(iii)$}\label{sec:dupes}

It is natural to consider uniqueness of the solution as the intersection of two events: the event that all solutions of the puzzle are similar, and the event that the puzzle does not contain duplicate or rotationally symmetric pieces. The former is characterized by parts $(i)$ and $(ii)$ of Theorem \ref{thm:main}, and, as stated in the introduction, I conjecture that it has a very sharp threshold at $q=\frac{2}{\sqrt{e}}n+o(n)$. It remains to consider the latter event.

\begin{proposition}\label{prop:dupes}
The probability that a random jigsaw puzzle contains either duplicate or rotationally symmetric pieces is $o(1)$ for $q=\omega(n)$.
\end{proposition}
\begin{proof}
Let $X$ denote the number of pairs of duplicate jigsaw pieces, and $Y$ the number of rotationally symmetric pieces respectively. The probability that two given jigsaw pieces have identical jig types is  $\Theta(q^{-4})$ if the pieces are non-adjacent in the planted assembly, and $O(q^{-3})$ if they are adjacent. Hence, $$\mathbb{E} X = \Theta(n^4 q^{-4})+O(n^2q^{-3}).$$ Similarly, the probability that a jigsaw piece has rotational symmetry is $q^{-2}$, which implies that $$\mathbb{E} Y = \Theta(n^2 q^{-2}).$$

By Markov's inequality it follows that $$\mathbb{P}(X+Y\neq 0) \leq  \Theta(n^4 q^{-4})+O(n^2q^{-3}) + \Theta(n^2 q^{-2}),$$
which tends to $0$ for $q=\omega(n)$. \end{proof}

By part $(ii)$ of Theorem \ref{thm:main}, all solutions of a random jigsaw puzzle are similar with high probability when $q=\omega(n)$. Part $(iii)$ of Theorem \ref{thm:main} follows by combining this with Proposition \ref{prop:dupes}. \qed

\begin{remark}
Considering the estimates for $\mathbb{E}X$ and $\mathbb{E} Y$ further, one would expect the probability of $X=Y=0$ to be bounded away from $0$ and $1$ for $q=\Theta(n)$. In particular, this would mean that the probability that a random jigsaw puzzle has a unique solution, up to global rotation, is bounded away from $0$ and $1$ when $(2+\varepsilon)n\leq q = O(n)$. For the sake of brevity, we will not attempt to prove this here. We can however note that a partial result to this effect was shown in Section 3 of \cite{NPS16}, namely that
$$\mathbb{P}(X=0) \leq \exp\left( -\frac{n^4-2n^2}{8q^4}\right),$$
which implies that the probability of a unique solution is bounded away from $1$ for $q=O(n)$. In particular, the bound $q=\omega(n)$ in part $(iii)$ of Theorem \ref{thm:main} is sharp.
\end{remark}

\section{Preliminaries for proof of Theorem 1.1, part $(ii)$}\label{sec:allsimilar1}

Let us start by defining some concepts. A \emph{complete assembly} of an $n\times n$ puzzle is a positioning and orientation of the jigsaw pieces into an $n\times n$ grid. We will formally consider this as a bijective map from $\{1, \dots, n\}^2$ to itself together with a map $\{1, \dots, n\}^2\rightarrow \Z_4$, interpreted as the position of and orientation of each piece relative to the planted assembly. Similarly, a \emph{partial assembly} is a positioning and orientation of a subset of pieces in the jigsaw puzzle into a square grid. We remark that an assembly itself is not random -- it just represents a reordering and rotation of (some of) the pieces, irrespective of how the pieces may be carved. One important example of a partial assembly is a \emph{$k\times k$ window} obtained by picking an appropriate square from a complete assembly.

We will consider a jigsaw piece to be a vertex with four cyclically ordered \emph{half-edges}, representing the sides of the piece. For a given (complete or partial) assembly $A$, we say that two half-edges are \emph{connected in $A$} if they correspond to sides of two different jigsaw pieces that are aligned next to each other in the assembly. Thus, any assembly $A$ can be considered as a graph by joining connected half-edges into edges. For an assembly $A$ and $\omega\in \Omega$, we say that \emph{$A$ is feasible with respect to $\omega$} if first assigning the planted assembly the carving $\omega$ and then reassembling the pieces according to $A$ (while keeping the carvings of the individual pieces) means that all connected pairs of half-edges get complementary jig types.

For any assembly $A$, we have a natural notion of a dual graph. Considering the assembly geometrically in the plane, the vertices in this graph are the common corners of at least two jigsaw pieces, and the edges are the common sides of two pieces. Hence the edges of the dual graph correspond to the connected pairs of half-edges in $A$. See Figure \ref{fig:dualgraph}.

\begin{figure}
\centering
\includegraphics[scale=.7]{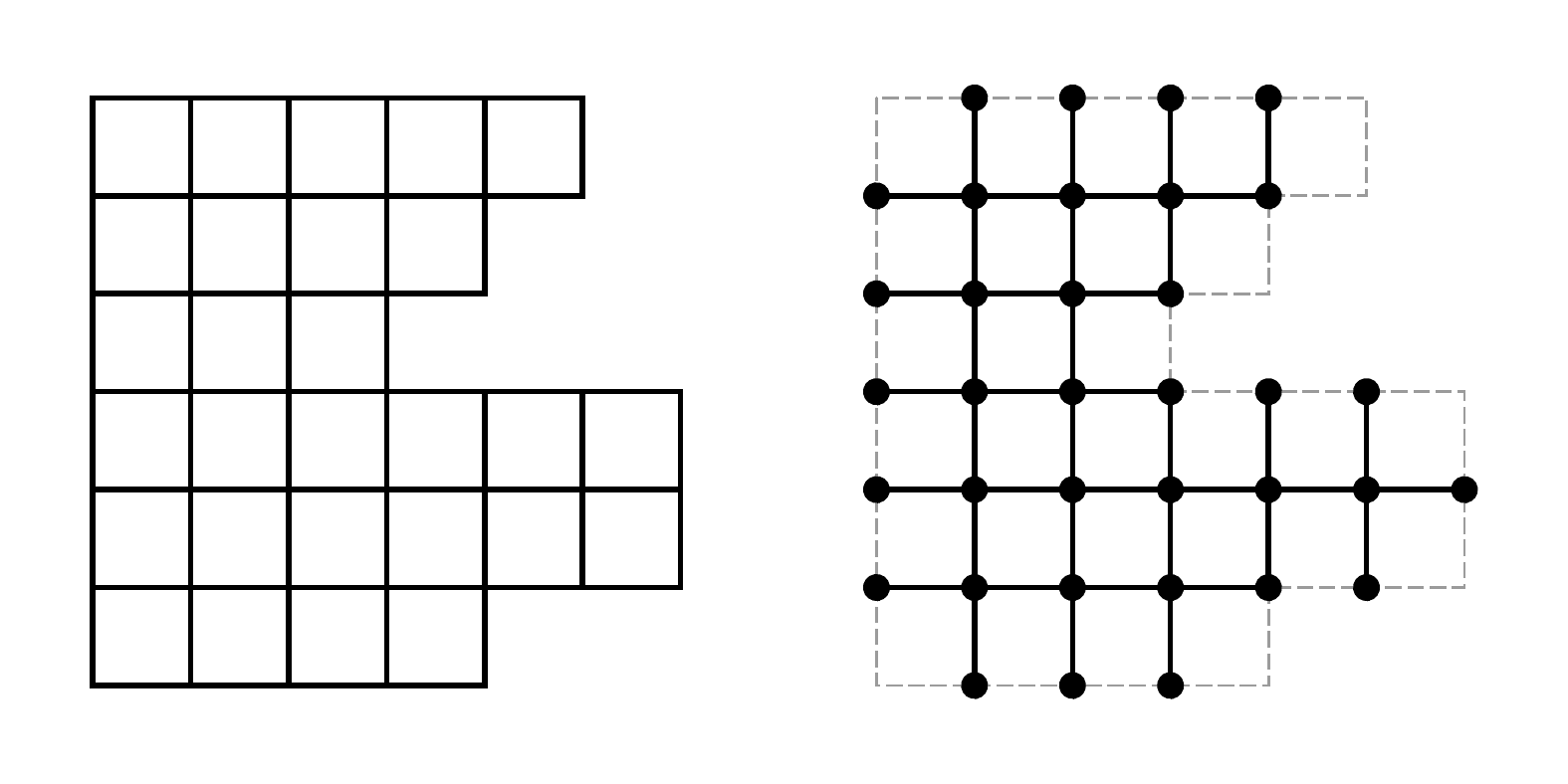}
\caption{\label{fig:dualgraph} Illustration of a partial assembly (left) and the corresponding dual graph (right).}
\end{figure}

The \emph{contour graph} of an assembly $A$, $C(A)$, is the subgraph of the dual graph of $A$ consisting of all edges whose corresponding pairs of half-edges are not connected in the planted assembly, see Figure \ref{fig:contour}. We will refer to the connected components of the contour graph as \emph{contours}. Note that the contour graph, together with the boundary of the assembly, partitions the pieces in $A$ into \emph{connected regions} given by the sets of pieces contained in each face. An important observation is that, within each of these regions, the positions and orientations of pieces differ from the planted assembly by a common translation and rotation, and an edge of the dual graph lies in the contour graph if and only if its half-edges come from different regions.

\begin{figure}
\centering
 \makebox[0pt]{\includegraphics[scale=1.15]{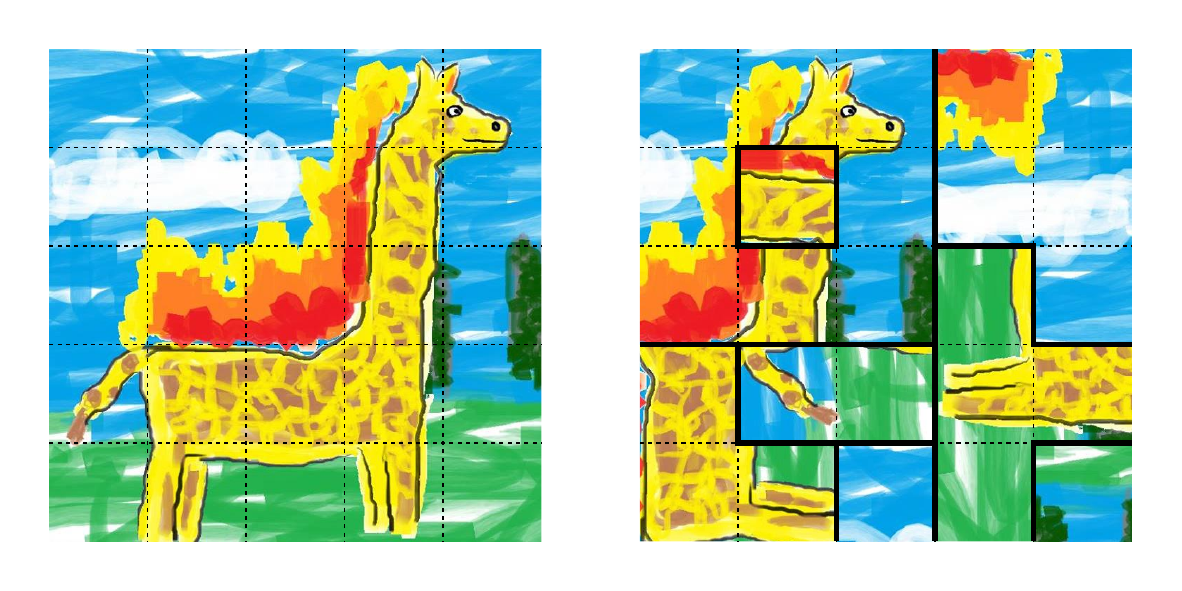}}
\caption{\label{fig:contour} An example of a contour graph. Though not part of the random model, the puzzle is here given a motif to better visualize the relative position of the jigsaw pieces. The left picture illustrates the planted assembly of a $5\times 5$ puzzle, and the right one an alternative assembly. Edges of the contour graph are marked by bold line segments. Here, the contour graph contains three contours, and partitions the puzzle into $8$ connected regions.}
\end{figure}

One important idea of the proof of part $(ii)$ of Theorem \ref{thm:main} is to limit the effect the planted assembly may have on making other solutions more likely than they would be with the simple independent model mentioned in the end of the introduction. Here the dependency effect is going to be limited by Lemma \ref{lemma:cfi} below, which relies on not too many edges in a proposed alternative assembly being of the same type as each other. This is clearly not possible for complete assemblies as are $2n(n-1)$ interior edges, but only $q=O(n)$ edge types. Hence, this will require us to restrict attention to small partial assemblies of the puzzle. What follows are two rather technical definitions whose point is to capture this notion.

Given an assembly $A$, and a carving $\omega\in\Omega$ by which $A$ is feasible, we say that an edge in the dual graph of $A$ has \emph{type $\{j, \iota(j)\}$} (with respect to $\omega$) if the half-edges across it have jig shapes $j$ and $\iota(j)$. Let $\mathcal{J}=\{ \{j, \iota(j)\} : j=1, 2,\dots, q\}$ be the set of between $\frac{q}{2}$ and $q$ possible types of edges in the dual graph. For any such pair $A, \omega$ and a set $\mathcal{E}$ of edges in the dual graph of $A$, the \emph{shape multiplicity of $\mathcal{E}$} with respect to $\omega$ is defined as
\begin{equation*}
sm(\mathcal{E}, \omega) = \sum_{J\in\mathcal{J} } \lfloor \#\{\text{edges in $\mathcal{E}$ of type $J$}\}/2\rfloor.
\end{equation*}
Similarly, the shape multiplicity of the assembly, $sm(A, \omega)$, is the shape multiplicity of its dual graph. Note that shape multiplicity can be interpreted as the maximal number of disjoint pairs of edges of the same type in $\mathcal{E}$. 

For any complete assembly $A$ and an $\omega\in\Omega$ by which $A$ is feasible, we say that $A$ is \emph{$k$-good} (with respect to $\omega$) if the shape multiplicity of any $k\times k$ window is at most $1$ when the window touches the boundary of the puzzle, and at most $2$ otherwise.

\begin{proposition}\label{prop:kgood}
For $k=o(n^{1/12})$ and $q=\Omega(n)$, the planted assembly is $k$-good with high probability.
\end{proposition}
\begin{proof}
In order for a $k\times k$ window of the planted assembly to have shape multiplicity at least $i+1$, there must exist $i+1$ disjoint pairs of edges in the dual graph where each pair has a common type. By the union bound, the probability that this occurs in a carving of a $k\times k$ window is $O( k^{4(i+1)} q^{-(i+1)})$. Summing this over $O(n)$ windows with $i=1$ and $O(n^2)$ ones with $i=2$ yields $O(n \,k^{2\cdot 4} q^{-2} + n^2 k^{2\cdot 6} q^{-3})=o(1)$.
\end{proof}
\begin{proposition}\label{prop:onlykgood}
For $k=o(n^{1/12})$ and $q=\Omega(n)$, with high probability, either all feasible complete assemblies are similar or there are (at least two) non-similar feasible $k$-good complete assemblies.
\end{proposition}
\begin{proof}
Let us denote the event described in the proposition by $E$. Condition condition on the complement of $E$, that is, the event that all feasible $k$-good complete assemblies are similar, but there is at least one more feasible complete assembly that is not $k$-good. Consider the box of the puzzle (that is, the unordered collection jigsaw pieces). There are at most four $\omega\in\Omega$ that gives rise to this box, namely those corresponding to the four rotations of the feasible $k$-good complete assemblies, and at least one additional $\omega\in\Omega$, corresponding to the non $k$-good assembly. As the carving of the planed assembly is chosen uniformly at random, each of these $\geq 5$ carvings are equally likely. Hence, with probability at least $\frac{1}{5}$, the planed assembly is not $k$-good. This means that $\left(1-\mathbb{P}(E) \right) \cdot \frac{1}{5} \leq \mathbb{P}(\text{planted assembly is not $k$-good})$, where the right-hand side is $o(1)$ by Proposition \ref{prop:kgood}. Hence, $\mathbb{P}(E)=1-o(1)$, as desired.
\end{proof}

We now turn using shape multiplicity for estimating the probability that a given partial assembly $A$ is feasible. If one were to disregard any dependency effects from the planted assembly, one would expect the probability that $A$ is feasible to be $q^{-E}$ where $E$ is the number of edges in the contour graph of $A$. However, this is not generally true as events of the form \emph{pairs of half-edges fit together} may not be independent. In fact, in the case of an edge-matching puzzle, that is, $\iota=id$, one can show that such events are always increasing in the sense that $\mathbb{P}(A_m \vert \cap_{i=1}^{m-1} A_i) \geq \mathbb{P}(A_m)$ where $A_1, A_2, \dots A_m$ each denote the event that some assembly is feasible, and this inequality can be made strict by choosing appropriate assemblies.

\begin{lemma}\label{lemma:cfi}
Fix a partial assembly $A$ and let $E$ denote the number of edges in its contour graph $C(A)$.  If no two half-edges across $C(A)$ are connected in the planted assembly, then
$$\mathbb{P}(A\text{ is feasible}) = q^{-E}.$$
Moreover, without this restriction we have that for any $i \geq 0$,
$$\mathbb{P}(A\text{ is feasible} \wedge sm(C(A), \omega)\leq i ) \leq q^{-E+i}.$$
\end{lemma}
\begin{proof}
We construct the abstract graph $G$ whose vertices are the half-edges across in $C(A)$. We connect a pair of vertices by an \emph{old edge} if they are connected in the planted assembly, and by a \emph{new edge} if they are connected in the partial assembly in the statement of the lemma. We will refer to such pairs of vertices in $G$ as \emph{old pairs} and \emph{new pairs} respectively. Note that each vertex in $G$ is the end-point of exactly one new edge and at most one old edge, so $G$ consists of paths and cycles. Further, these components alternate between old and new edges, and, in the case of a path, begin and end with new edges.

The way the carving of the planted assembly is chosen means that each vertex in $G$ that is not part of an old pair is independently and uniformly assigned a jig type, and each old pair is independently and uniformly assigned a pair of complementary jig types. Further, the new assembly is feasible if the jigs of half-edges in each component alternate between two complementary types. We say that the component is \emph{feasible} if this holds.

In the case where no two half-edges across the contour graph are connected in the planted assembly, there are no old edges in $G$. Hence each of the $E$ new pairs fits together independently with probability $\frac{1}{q}$.

Considering the latter case, a path in $G$ with $E'$ new edges, $$h_1 \rightarrow h_2 \rightarrow \dots \rightarrow h_{2E'},$$ is feasible if  $h_2, h_4, \dots, h_{2E'}$ all have the complementary jig type to $h_1$, which occurs with probability $q^{-E'}$. Similarly, a cycle in $G$ with $E'$ new edges, $$h_1 \rightarrow h_2\rightarrow \dots\rightarrow h_{2E'}\rightarrow h_1,$$ is feasible if $h_2, h_4, \dots, h_{2E'}$ all have the same jig type, which occurs with probability $q^{-E'+1}$. As the total number of new edges is $E$, we see that the probability that the new assembly is feasible is $q^{-E}$ when $G$ only consists of paths, and increases by a factor of $q$ for each cycle.

Note that, by the definition of contour graph, a pair of vertices in $G$ cannot be connected by both an old and a new edge. Hence, each cycle contains at least two new edges. As a consequence, for each feasible cycle in $G$ there are two edges of the same type in $C(A)$. Hence, if the partial assembly is feasible, $sm(C(A), \omega)$ is at least the number of cycles in $G$. In particular, there are either more than $i$ cycles, in which case the probability of the event in the statement is $0$, or there are at most $i$ cycles, in which case the probability is at most $q^{-E+i}$, as desired.
\end{proof}

We further need to bound the number of possible shapes of contours in assemblies. To this end, we have the following estimate.

\begin{lemma}\label{lemma:euler}
Up to translation, the number of connected subgraphs $G\subset \mathbb{Z}^2$ with $E$ edges and $F$ bounded faces is at most ${3E-4F+4 \choose 2E-4F+4}$. Moreover, for any $\varepsilon'>0$, there exists a $M=M_{\varepsilon'}>0$ such that
$${3E-4F+4 \choose 2E-4F+4} = O_{\varepsilon'}\left( M^{E-2F} (1+\varepsilon')^{E}\right),$$
where the subscript $\varepsilon'$ denotes that $M$ as well as the implicit constant in the big O-notation are allowed to depend on $\varepsilon'$, but nothing else.

\end{lemma}
\begin{proof}
For any such $G$, let $V$ denote its number of vertices. Then, by the Euler characteristic formula, we have
$$V-E+F = 1.$$
Note in particular that this means that $V$ is independent of the choice of $G$. It further follows that
$$R := \sum_{v\in G} \left(4 - deg(v)\right) = 4V-2E = 2E - 4F + 4.$$

Now, consider the following procedure for constructing $G$. Since we are only interested in the number of graphs up to translation, we may, without loss of generality, assume that the origin is a vertex in $G$. Let $G_0$ be the empty graph, and let $G_1$ be the graph only containing the origin. For each $i\geq 1$, we construct $G_{i+1}$ by selecting a subset of the edges between $V(G_{i})\setminus V(G_{i-1})$ and $\Z^2\setminus V(G_{i})$, and adding them to $G_i$. If $G_{i+1}=G_i$, then we terminate and put $G=G_i$.

In order for the procedure to construct a graph with $E$ edges dividing the plane into $F$ regions, it must choose to include an edge $E$ times, and choose not to do so $R$ times. Hence we can encode any such graph as a binary string of $E$ ones and $R$ zeroes, which can be done in ${E+R \choose R}$ ways.

Finally, considering the double sum
$$\sum_{i=0}^\infty \sum_{j=0}^\infty {i+j \choose j} x^{i}y^{j} = \frac{1}{1-x-y},$$
which converges absolutely for $\abs{x}+\abs{y}<1$, it follows that for any $i, j\in \mathbb{N}_0$ and any $x, y\geq 0$ such that $x+y<1$ we have
$$ {i+j \choose j} \leq \frac{ x^{-i} y^{-j} }{1-x-y} = O_{x,y}( x^{-i}y^{-j}).$$
Letting $i=E$, $j=2E-4F+4$, $x=(1+\varepsilon')^{-1}$ and choosing $M=M_{\varepsilon'}$ sufficiently large so that $y=M^{-1/2}<1-x$, we get
$$ {3E-4F+4 \choose 2E-4F+4} = O_{\varepsilon'} \left( (1+\varepsilon')^{E} M^{E-2F+2}\right)=O_{\varepsilon'} \left( (1+\varepsilon')^{E} M^{E-2F}\right).$$
\end{proof}

\section{Proof of Theorem 1.1, part $(ii)$}\label{sec:allsimilar2}

Let us once and for all fix the value of $\varepsilon>0$ such that $q\geq (2+\varepsilon)n$, and choose a function $k=k(n)$ such that $k$ is always even, $k=\omega(\ln n)$ and $k=o(n^{1/12})$, say $k=2\lceil(\ln n)^2\rceil$. The strategy for the rest of the proof will be as follows. By Propositions \ref{prop:kgood} and \ref{prop:onlykgood} it suffices to show that, with high probability, the puzzle has no $k$-good feasible complete assemblies except (possibly) ones similar to the planted assembly. Let $\mathcal{A}$ denote the set of all complete assemblies whose contour graphs are non-empty and do not contain any contours consisting of exactly four edges that surrounds a single jigsaw piece. We know that if, for some $\omega\in\Omega$, there exists a $k$-good feasible complete assembly $A$ which is not similar to the planted assembly, then, by possibly moving around identical pieces, we may assume that $A\in\mathcal{A}$. Hence it remains to show that, with high probability, no $A\in\mathcal{A}$ is both feasible and $k$-good. This is divided into three cases, as described below, depending on $C(A)$. We have
\begin{equation*}
\mathbb{P}(\exists k\text{-good feasible }A\in\mathcal{A}) \leq \sum_{i=1}^3 \mathbb{P}(\exists k\text{-good feasible }A\in\mathcal{A}\text{ satisfying Case }i).
\end{equation*}
For each case $i$, we show that, for any $\omega\in\Omega$, the existence of a $k$-good feasible $A\in\mathcal{A}$ (with respect to $\omega$) implies the existence of a certain kind of partial assembly $A'\in\mathcal{A}'_i$ which is feasible and has $sm(A', \omega)$ at most $2$ for $i=1$ or $2$ and at most $1$ for $i=3$. Using Markov's inequality we show that, with high probability, no $A'\in\mathcal{A}_i'$ has these properties. This completes the proof of part $(ii)$ of Theorem \ref{thm:main}.
\\~\\
\textbf{Case 1:} There exists a contour not intersecting the boundary of the assembly, where the vertical and horizontal distances between pairs of vertices are each at most $k-2$.\\

Suppose that for a certain $\omega\in\Omega$ the contour graph of a $k$-good feasible complete assembly $A\in\mathcal{A}$ whose contour graph contains such a contour $C$. Note that $C$ can be contained in a $k\times k$ window, hence $sm(C, \omega)\leq 2$. By taking the subset of jigsaw pieces that have at least one corner on $C$, we obtain a feasible partial assembly with contour graph $C$ that satisfies $sm(C, \omega)\leq 2$. This is illustrated in Figure \ref{fig:case1}.

\begin{figure}
\centering
\includegraphics[scale=.5]{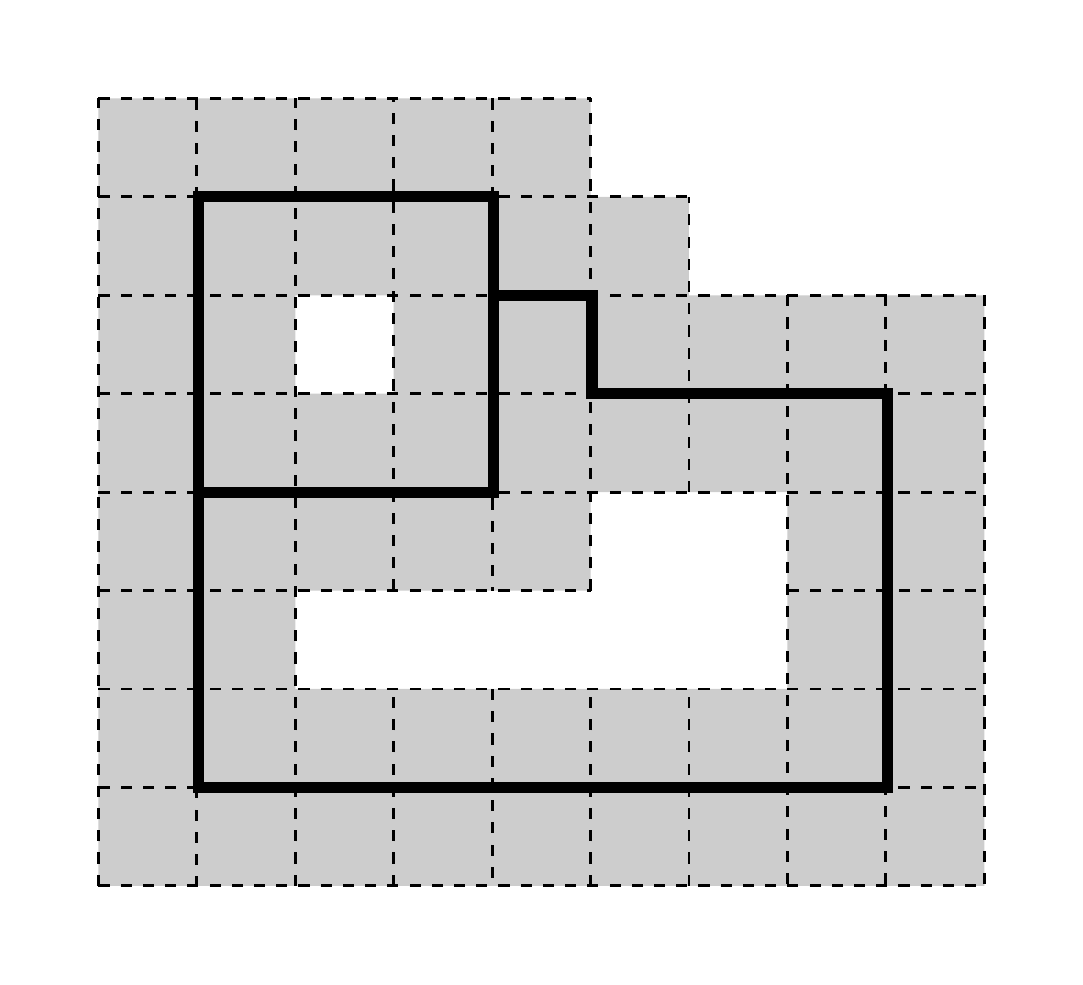}
\caption{ \label{fig:case1} Example of a minimal partial assembly around a contour. Here, the contour has two bounded faces, partitioning the assembly into three connected regions.}
\end{figure}

Let us start by fixing a connected graph $C\subset \mathbb{Z}^2$, and bounding, among the corresponding partial assemblies, the expected number of feasible ones. We denote the number of edges of $C$ by $E$ and the number of bounded faces by $F$. Without any restrictions on whether half-edges across the contour should be connected in the planted assembly, there are at most $(4n^2)^{F+1}$ ways to choose the original positions and orientations of the $F+1$ connected regions in the planted assembly. In counting the number of these where some pair of half-edges across $C$ are connected in the planted assembly, we can first choose a pair of such half-edges. This relates the position and orientation of two connected regions, hence the positions and orientations of connected regions can be chosen in at most $(4n^2)^{F}$ ways. By Lemma \ref{lemma:cfi}, it follows that, for this fixed $C$, the expected number of corresponding feasible assemblies is at most
$$(4n^2)^{F+1} q^{-E} + {2E \choose 2} (4n^2)^{F} q^{-E+2} \leq O(E^2) (2n)^{2F-E+2}\left(1+\frac{\varepsilon}{2}\right)^{-E}.$$
Note that $C$ is contained in a $k\times k$ window, hence $E=O(k^2)$ and we can replace $O(E^2)$ by $O(k^4)$ in the right hand side.

Summing this over all possible $C$, we apply Lemma \ref{lemma:euler} with $\varepsilon'$ chosen such that $(1+\varepsilon')(1+\frac{\varepsilon}{2})^{-1} \leq (1+\frac{\varepsilon}{3})^{-1}$. This gives an upper bound of 
$$O_\varepsilon(n^2 k^4) \sum_E \left(1+\frac{\varepsilon}{3}\right)^{-E} \sum_F \left(\frac{M}{2n}\right)^{E-2F},$$
where $E$ and $F$ goes over all possible combinations of numbers of edges and faces. Making the substitution $T=E-2F$ and letting $T_0$ denote the minimal possible value of $E-2F$, we can bound the sum by
$$O_\varepsilon(n^2 k^4) \sum_{E=0}^\infty \left(1+\frac{\varepsilon}{3}\right)^{-E} \sum_{T=T_0}^\infty \left(\frac{M}{2n}\right)^{T} = O_\varepsilon(n^{2-T_0} k^4).$$

To estimate $T_0$, consider the sum of perimeters of all bounded faces of such a contour $C$. On the one hand, this will count every interior edge in $C$ twice, and every edge on the boundary once, hence the sum equals $2E-P$ where $P$ is the perimeter of $C$. On the other hand, each face has perimeter at least four, hence the sum is at least $4F$. We conclude that $E-2F \geq \frac{P}{2}$, and as we assumed that $C$ is not just four edges surrounding a single jigsaw piece, we must have $P>4$ and thus $E-2F >2$ Hence $T_0\geq 3$, which implies that the expected number of such assemblies tends to $0$ as $n\rightarrow\infty$, as desired.
\\~\\
\textbf{Case 2:} No contour satisfies Case 1, but there is a contour with at least one vertex at (geometrical) distance $\geq \frac{k}{2}$ from the boundary.\\

Again, let $\omega\in\Omega$ and suppose there is a $k$-good feasible complete assembly $A\in\mathcal{A}$ with a contour of this form. Consider a $k\times k$ window centered at some vertex meeting the conditions of Case 2 on this contour. Let $C$ be the contour graph of the partial assembly consisting of this window, and let $S$ be the surrounding square in the dual graph. Note that since the assembly is $k$-good, $sm(C, \omega)\leq 2$, and as no contour in the complete assembly satisfies Case 1, all contours of $C$ must reach the boundary of the $k\times k$ window. Hence if such a contour exists, then there is a connected graph $C\cup S$ whose boundary $S$ is a square of side length $k$ and such that the midpoint of the square is a vertex in $C$, together with a feasible partial assembly in this square with contour graph $C$ such that $sm(C, \omega)\leq 2$.

Fix such a pair $C, S$, and let $E$ denote the number of edges of $C$, and $F$ the number of bounded faces of $C\cup S$. By the same argument as in Case 1, the expected number of corresponding partial assemblies that are feasible is at most
$$(4n^2)^{F}q^{-E} + {2E \choose 2}(4n^2)^{F-1} q^{-E+2} = O(k^4) (2n)^{2F-E} \left(1+\frac{\varepsilon}{2}\right)^{-E}.$$
Summing this over all possible $C$ using Lemma \ref{lemma:euler} we get an upper bound of
\begin{equation}\label{eq:case2} O_\varepsilon(k^4 M^{4k}) \sum_{E} \left(1+\frac{\varepsilon}{3}\right)^{-E} \sum_F \left(\frac{M}{2n}\right)^{E-2F},\end{equation}
where, again, $E$ and $F$ runs over all possible combinations of numbers of edges and faces.

\begin{lemma}\label{lemma:nobigbox}
Let $C, S$, $E$ and $F$ be as above. Then $E-2F \geq \frac{k}{4}-2-O\left(\frac{E}{k}\right)$.
\end{lemma}
\begin{proof}
Enumerate the bounded faces of $C\cup S$ from $1$ to $F$ and let $P_i$ denote the perimeter of the $i$:th face. Then, as every edge in $C$ is counted in the perimeter of two faces, and every edge in $S$ in one face, we get
\begin{equation}\label{eq:case2sumperimeter}
\sum_{i=1}^F P_i = 2E+4k.
\end{equation}
Note that $2E+4k$ here corresponds to $2E-4k$ in Case 1 because there $E$ does not include edges on the boundary of the $k\times k$ square. 

The perimeter of each face can be bounded in terms of its area. Suppose face $i$ has area $A_i$, and that it contains jigsaw pieces from $w_i$ columns and $h_i$ rows. Then, on the one hand $A\leq w_i\cdot h_i$, and on the other hand the face contains two horizontal and vertical edges for each of these columns and rows. Hence $$P_i\geq 2w_i+2h_i \geq 2w_i + 2 A_i/w_i \geq 4\sqrt{A_i}.$$

For $A_i > \frac{15}{16}k^2$ we can improve this lower bound using the fact that the midpoint of the square is a vertex in $C$. Let us consider $C\subseteq \mathbb{Z}^2$ such that the midpoint is $(0, 0)$, that is, the square is the area $[-\frac{k}{2}, \frac{k}{2}]^2$. Let $\Gamma$ be a minimal length path from $(0, 0)$ to the boundary of $[-\lceil \frac{k}{4}\rceil, \lceil \frac{k}{4}\rceil]^2$. We may, without loss of generality, assume that $\Gamma$ hits the boundary on the line $x=\lceil \frac{k}{4}\rceil$. Note that this means that the $y$-coordinate of any vertex along $\Gamma$ lies between $-\frac{k}{4}$ and $\frac{k}{4}$. Then, in each column $[l, l+1]\times \R$ for $l=0, 1, \dots, \lceil \frac{k}{4}\rceil-1$ there are either at least $4$ horizontal edges on the boundary of the face, or at most $\frac{3k}{4}$ area units inside the face. Letting $x$ be the number of columns satisfying the former, and $y$ the number that satisfies the latter, we get
$$P_i \geq 4\sqrt{A_i} + 2x,$$
$$A_i \leq k^2-\frac{k}{4} y,$$
$$x+y \geq \frac{k}{4}.$$
Hence,
$$P_i \geq 4\sqrt{A_i}+2x \geq 4\sqrt{A_i}+\frac{k}{2}-2y \geq 4\sqrt{A_i}+\frac{8}{k}\left(A_i-\frac{15}{16}k^2\right).$$

Let $f(a) = 4\sqrt{a}+\frac{8}{k} \max(0, a-\frac{15}{16}k^2)$. This function is continuous, and concave on the intervals $[0, \frac{15}{16}k^2]$ and $[\frac{15}{16}k^2, \infty]$. Then, by \eqref{eq:case2sumperimeter} and for a given $F$, the value of $2E+4k$ is bounded from below by the minimal value of $\sum_{i=1}^F f(a_i)$ subject to $\sum_{i=1}^F a_i = k^2$ and $a_i\geq 1$ for $i=1, 2, \dots, F$.

If $F-1 \geq \frac{1}{16}k^2$, we have $a_i\leq \frac{15}{16}k^2$ and hence the minimization problem is concave. In this case the function attains its minimum at an extreme point on the boundary of the domain. Up to permutation of variables, there is only one such point, namely $a_1=k^2-(F-1)$ and $a_2=a_3=\dots=a_F=1$, yielding a minimum of $4\sqrt{k^2-(F-1)} + 4(F-1) = 4k + 4(F-1) - O\left(\frac{F}{k}\right)$.

On the other hand, if $F-1 < \frac{1}{16}k^2$, we can divide the problem into two concave minimization problems by considering the case where all $a_i$:s are less than $\frac{15}{16}k^2$ and the case where at least one variable, say $a_1$, is at least $\frac{15}{16}k^2$. Again, up to permutation of variables, the only extremal points are
$$a_1=\frac{15}{16}k^2, a_2 = \frac{1}{16}k^2-(F-2), a_3=a_4=\dots=a_F=1,$$
$$a_1=k^2-(F-1), a_2=a_3=\dots=a_F=1,$$
with the corresponding values
$$4\sqrt{\frac{15}{16}k^2} + 4\sqrt{\frac{1}{16}k^2-(F-2)}+4(F-2) = (\sqrt{15}+1)k + 4(F-2)-O\left(\frac{F}{k}\right),$$
and
$$4\sqrt{k^2-(F-1)}+\frac{8}{k}\left( \frac{1}{16}k^2-(F-1)\right)+4(F-1) = \frac{9}{2}k+4(F-1)-O\left(\frac{F}{k}\right).$$
Hence, the minimum in this case is $\frac{9}{2}k+4(F-1)-O(\frac{F}{k})$.

Note that for $F-1 \geq \frac{1}{16}k^2$, we have $\frac{1}{2}k = O(\frac{F}{k})$. Hence, for all $F$, we can write the minimum as $\frac{9}{2}k + 4(F-1)-O(\frac{F}{k})$.

In conclusion, we have $2E+4k \geq \frac{9}{2}k+4(F-1) - O(\frac{F}{k})$. Hence $E-2F \geq \frac{1}{4}k-2-O(\frac{F}{k})$, where clearly $F=O(E)$.
\end{proof}

Using Lemma \ref{lemma:nobigbox}, we can bound \eqref{eq:case2} by
\begin{align*}
O_\varepsilon(k^4 M^{(4+\frac{1}{4})k-2} n^{2-\frac{k}{4}}) \sum_{E=0}^\infty \left(\frac{ (2n/M)^{O(\frac{1}{k})}}{1+\frac{\varepsilon}{3}}\right)^{E}.
\end{align*}
As $k=\omega(\ln n)$, $(2n/M)^{O(\frac{1}{k})}\rightarrow 1$ as $n\rightarrow\infty$. Hence the sum is at most $$O_\varepsilon(k^4 M^{(4+\frac{1}{4})k-2} n^{2-\frac{k}{4}}),$$ which tends to $0$ as $n\rightarrow\infty$. Again, we can conclude that, with high probability, no such partial assemblies exist.
\\~\\
\textbf{Case 3:} All vertices in the contour graph have distance $<\frac{k}{2}$ from the boundary.\\

For any such $A\in\mathcal{A}$ there is a large connected region in the assembly that contains the $(n-2k)\times(n-2k)$ square of all jigsaw pieces at distance at least $k$ from the boundary. Consider the area in the planted assembly that corresponds to this square. In order to fit in the planted assembly, this area must cover all jigsaw pieces that are not in the $2k$ outermost layers. As a consequence of this, any jigsaw piece in the $k$ outermost layers of $A$ must come from the $2k$ outermost layers of the planted assembly.

Consider any $3\times 3$ window in $A$ with a non-trivial contour graph $C$. Let again $S$ be the boundary of the window, $E$ the number of edges of $C$, and $F$ the number of bounded faces of $C\cup S$. As this window is at distance $\leq k$ from the boundary, it is part of a $k\times k$ window that touches the boundary. Hence, by the definition of $k$-good, $sm(C, \omega)\leq 1$. Furthermore, by the above reasoning, all connected regions come from the $2k$ outermost layers of the planted assembly.

As before we apply Markov's inequality to the number of such partial assemblies that can be constructed from the random jigsaw puzzle. There are at most $4\cdot 8kn$ ways to choose the original orientation and position of each connected region, hence, by Lemma \ref{lemma:cfi}, the expected number of such assemblies corresponding to a fixed $C$ is at most
$$ (32kn)^F q^{-E} + O(1) (32 kn)^{F-1} q^{-E+1} = \left(1+O\left(\frac{1}{k}\right)\right) (32 kn )^F O(n^{-E}).$$
We note that if $E>F$, this is bounded by $O(k^F n^{F-E})=O(\frac{k^9}{n})\rightarrow 0$ as $n\rightarrow\infty$. Hence, with high probability, the only $3\times 3$ windows that can occur in any $k$-good feasible $A\in\mathcal{A}$ are those where either the contour graph is empty, or possibly those where the contour graph satisfies $E\leq F$.

\begin{lemma}
The only possible $C$ where $E\leq F$ are the ones where $C$ consists of exactly two edges that separate a corner from the rest of the $3\times 3$ window.
\end{lemma}
\begin{proof}
It is clear that the only possible way for $C\cup S$ not to be connected is if $C$ consists of four edges surrounding the center jigsaw piece, which by choice of $A$ is not possible. If $C\cup S$ is connected, then its Euler characteristic gives that $E-F = V-13,$ where $V$ denotes the number of vertices of $C\cup S$. Note that there are always $12$ vertices on the boundary of $C\cup S$, and, since $C$ is not empty, it must contain at least one additional vertex. Hence, the only possibility for $E\leq F$ is when $V=13$, which clearly only happens if $C$ consists of two edges that separate a corner piece from the rest of the $3\times 3$ window.
\end{proof}

Now, consider an $A\in\mathcal{A}$ with the property that the contour graph of each $3\times 3$ window is either empty, or consists of two edges that cut out a corner piece. In the latter case, if we can shift the window one step either vertically or horizontally towards that corner piece, then we would obtain a window with some other non-trivial contour graph. Clearly, the only case where this would not lead to a contradiction is if the cut out piece is one of the four corner pieces of $A$. Hence $A$ can only differ from the planted assembly by a global rotation and reordering of the the four corner pieces.

In conclusion, with high probability, no $A\in\mathcal{A}$ is both feasible and $k$-good except possibly those that only differ from the planted assembly by the positions of the four corner pieces. But, with high probability, the 16 sides of the corner pieces all have different jig shapes, and thus no such reordering is feasible either, as desired. \qed

\section*{Acknowledgement}
This research was partially supported by a grant from the Swedish Research Council. I thank my former supervisor, Peter Hegarty, and the anonymous referee for their thorough reading and valuable comments.

\section*{Note}

After this article was submitted, Balister, Bollob\'{a}s, and Narayanan \cite{BBN17} announced independent work on a very closely related model. Their main result is analogous to Theorem \ref{thm:main} apart from a less optimized constant.


\begin{thebibliography}{9}

\bibitem{BBN17} Balister, P., Bollob\'as, B., Narayanan, B., \textit{Reconstructing random jigsaws} (2017+), available at \href{http://arxiv.org/abs/1707.04730}{http://arxiv.org/abs/1707.04730}

\bibitem{BFM16} Bordenave, C., Feige, U., and Mossel, E., \textit{Shotgun assembly of random jigsaw puzzles} (2016+), available at \href{http://arxiv.org/abs/1605.03086}{http://arxiv.org/abs/1605.03086}

\bibitem{BDDHMY16+} Bosboom, J., Demaine, E.D., Demaine, M.L., Hesterberg, A., Manurangsi, P., and Yodpinyanee, A., \textit{Even $1\times n$ Edge-Matching and Jigsaw Puzzles are Really Hard}, Journal of Information Processing 25 (2017), 682--694.

\bibitem{DD07} Demaine, E.D. and Demaine, M.L. \textit{Jigsaw Puzzles, Edge Matching, and Polyomino Packing: Connections and Complexity}, Graphs and Combinatorics 23 (2007), Suppl 1, 195--208

\bibitem{M16} Martinsson, A., \textit{Shotgun edge assembly of random jigsaw puzzles} (2016+), available at \href{http://arxiv.org/abs/1605.07151}{http://arxiv.org/abs/1605.07151}

\bibitem{MR15} Mossel, E. and Ross, N., \textit{Shotgun assembly of labeled graphs} (2015+), available at \href{http://arxiv.org/abs/1504.07682}{http://arxiv.org/abs/1504.07682}

\bibitem{NPS16} Nenadov, R., Pfister, P., and Steger, A., \textit{Unique reconstruction threshold for random jigsaw puzzles}, Chicago journal of Theoretical Computer Science (2017).


\bibitem{M07} Description of Eternity II release, PR Newswire, 16 February 2007, \href{http://www.prnewswire.co.uk/news-releases/eternity-puzzle-is-back-this-summer-with-a-us2-million-prize-for-the-first-person-to-find-a-solution-153539245.html}{http://www.prnewswire.co.uk/news-releases/eternity-puzzle-is-back-this-summer-with-a-us2-million-prize-for-the-first-person-to-find-a-solution-153539245.html}
\end{thebibliography}
\end{document}